\documentclass{amsart}
\usepackage{graphicx}
\usepackage{mathtools}
\usepackage{amssymb}
\usepackage[lined,boxed,commentsnumbered]{algorithm2e}

\newtheorem{theorem}{Theorem}[section]
\newtheorem{lemma}[theorem]{Lemma}
\newtheorem{prop}[theorem]{Proposition}
\newtheorem{cor}[theorem]{Corollary}

\theoremstyle{definition}

\theoremstyle{remark}
\newtheorem{remark}[theorem]{Remark}

\numberwithin{equation}{section}

\newcommand{\diff}{\mathrm{d}}

\DeclareMathOperator{\sgn}{sgn}
\DeclareMathOperator{\sq}{sq}
\DeclareMathOperator{\cq}{cq}
\DeclareMathOperator{\tq}{tq}

\DeclareMathOperator{\arcsq}{arcsq}

\newcommand{\Beta}{\mathrm{B}}

\renewcommand{\Beta}{\mathrm{B}}

\begin{document}

\title{Derivative polynomials and infinite series for squigonometric functions}

\author{Bart S. van Lith}
\address{ASML}
\email{bartvanlith@gmail.com}
\thanks{Author supported by.}



\date{\today.}


\begin{abstract}
All squigonometric functions admit derivatives that can be expressed as polynomials of the squine and cosquine. We introduce a general framework that allows us to determine these polynomials recursively. We also provide an explicit formula for all coefficients of these polynomials. This also allows us to provide an explicit expression for the MacLaurin series coefficients of all squigonometric functions. We further discuss some methods that can compute the squigonometric functions up to any given tolerance over all of the real line.
\end{abstract}

\maketitle

\section{Introduction: squigonometry}

Squigonometry is the study of imperfect circles, according to Robert D. Poodiack and William E. Wood \cite{wood_squigonometry}. The subject deals with geometric planar shapes that are, in some sense, generalizations of circles, i.e., shapes defined by
\begin{equation}\label{eq:squircle_definition}
    |x|^p + |y|^p = 1.
\end{equation}
We can consider such shapes for many values of $p$, though typically we constrain ourselves to $1\leq p < \infty$, with the $p=\infty$ case as a limit. The most common alternative value is $p=4$, and many of our examples will be for this case. Shapes described by \eqref{eq:squircle_definition} are sometimes called $p$-circles or squircles, as they look like a hybrid between squares and circles. In fact, for either $p=1$ or $p=\infty$, they are in fact squares. Ordinary trigonometry, dealing with circles, is the special case where $p=2$. See Figure~\ref{fig:squircles_plot} for some examples.

\begin{figure}
    \centering
    \includegraphics[width=0.5\linewidth]{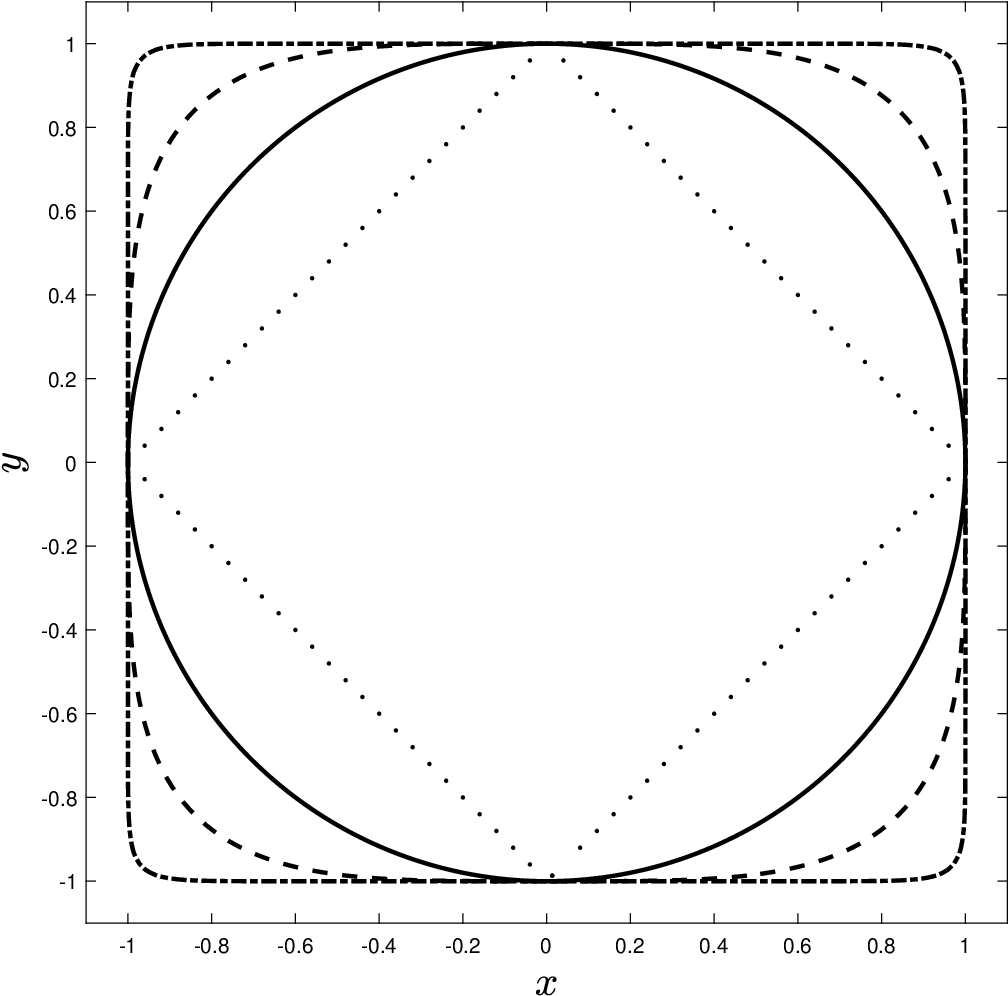}
    \caption{Several squircles for $p=1$ (dotted), $p=2$ (solid), $p=4$ (dashed), and $p=20$ (dash-dot).}
    \label{fig:squircles_plot}
\end{figure}

Just like trigonometric functions, it is possible to define the squigonometric functions that parametrize the $p$-circle. However, unlike trigonometry, there are various ways of doing so. For instance, we may consider parametrization by area, arc length, or angle. For the case $p=2$, these parametrizations all happen to coincide, while they do not for other $p$. We will employ the areal parametrization, which is probably the most common. We start by defining the inverse functions by means of an integral, i.e., the arcsquine is defined by
\begin{equation}\label{eq:arcsquine_integral}
    \arcsq(x) = \int\limits_0^x (1-u^p)^{\frac{1}{p}-1} \diff u.
\end{equation}
The case $p=2$ corresponds to the familiar arcsine function. There seems to be no standard notation or naming convention at the moment. We will adopt $\sq t$ and $\cq t$ and name them squine and cosquine respectively, with the order $p$ understood. In case the order is unclear, we shall append it as a subscript. We have plotted the squine and cosquine for $p=4$ in Figure~\ref{fig:squine_cosquine_plot}.

\begin{figure}
    \centering
    \includegraphics[width=0.7\linewidth]{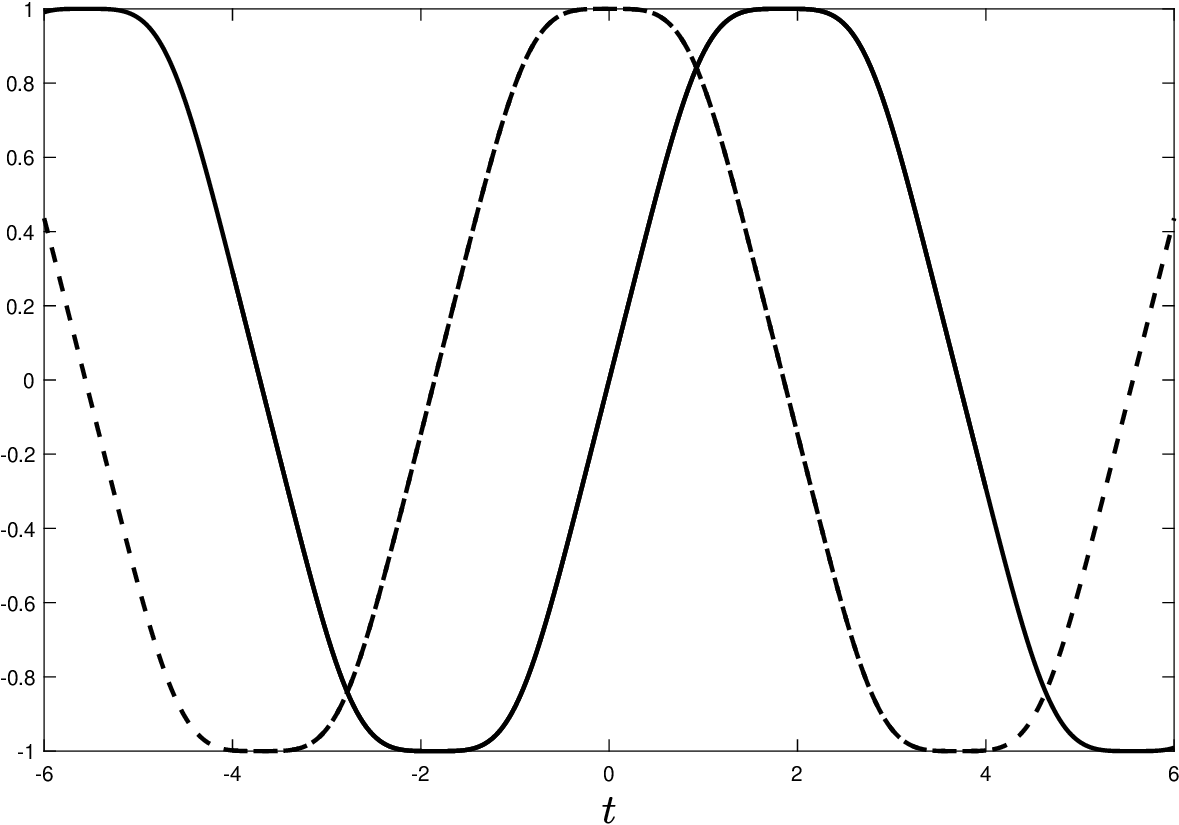}
    \caption{Squine (solid) and cosquine (dashed) for $p=4$.}
    \label{fig:squine_cosquine_plot}
\end{figure}

It is fairly easy to see that the integrand of \eqref{eq:arcsquine_integral} is positive on the interval $[0,1]$, so that we can actually define an inverse. The complete integral from $0$ to $1$ corresponds to a special number, defined as $\pi_p = 2 \arcsq(1)$, which is a generalization of $\pi$. The squine is, as a consequence, monotonic on $[0,\frac{\pi_p}{2}]$. The cosquine can then be defined as $\cq t = \sq(\frac{\pi_p}{2}-t)$. We simply define the squine and cosquine to satisfy the familiar shift and reflection identities such as $\sq(\pi_p-t) = \sq t$, etc. This allows us to extend their definition to all of $\mathbb{R}$. Defining the squine and cosquine on $\mathbb{C}$ is a tricky matter that we will not deal with here.

It is possible to show that in the first quadrant, the derivative of the squine and cosquine satisfy
\begin{subequations}\label{eq:squigonometric_derivatives}
    \begin{align}
        \frac{\diff}{\diff t} \sq t &= \cq^{p-1} t,\\
        \frac{\diff}{\diff t} \cq t &= -\sq^{p-1} t.
    \end{align}
\end{subequations}
Thus, as an alternative, we may define them to satisfy this set of differential equations, together with the initial condition $\sq 0 = 0$, $\cq 0 = 1$. Being parametrizations of the squircle, the squine and cosquine satisfy the Pythagorean identity
\begin{equation}\label{eq:pythagorean_identity}
    |\sq t|^p + |\cq t|^p = 1.
\end{equation}
If $p$ is an even integer, or we are working in the first quadrant, we can ignore the absolute values.

The ordinary sine and cosine exhibit power series that are not too hard to obtain by a variety of ways. For instance, calculating all derivatives is a possibility, so that we can apply Taylor's Theorem. Another method is to use Euler's identity for the complex exponential, i.e., $e^{i\theta}=\cos \theta + i\sin \theta$. The exponential function has, arguably, one of the simplest power series, so that the power series for sine and cosine are very easily obtained. The power series of the squine and cosquine are less simple, for instance, we have
\begin{equation}
    \sq t = t - \frac{p-1}{p(p+1)}t^{p+1} + \frac{(p-1)(2p^2-3p-1)}{2p^2(p+1)(2p+1)}t^{2p+1} + \ldots
\end{equation}
Clearly, the nonzero power series coefficients for the squine occur at powers $1+pj$ for $j=0,\ldots,\infty$. According to Poodiack and Wood, finding the coefficients of the MacLaurin series is somewhat of an open problem. Until now, only general methods such as series reversion techniques have proven effective. It is this open problem that has motivated this work.

Unlike trigonometric functions, \textit{squigonometric functions have finite radius of convergence}. The simplest expression for this radius is probably given by $R_p = \frac{\pi_p}{4}\sec(\frac{\pi}{p})$. Thus, $R_2 = \infty$, but, $R_4$ is roughly $1.31$. Moreover, it is not hard to show that $R_p \downarrow 1$ as $p \to \infty$. Since $\frac{\pi_p}{4}<1$, we can therefore extend the squigonometric functions to the entire real line by symmetry and periodicity, which we will use further on.

The means by which we find the MacLaurin series is to define derivative polynomials for general squigonometric functions. This concept has been introduced by Hoffman to find derivatives of the tangent and secant \cite{hoffman1995, hoffman1999}. He noticed, for instance, that all derivatives of the tangent can be written as polynomials in the tangent. These polynomials can then be calculated using very simple integer recursions. Boyadzhiev later generalized these results and found a connection to the Stirling numbers \cite{boyadzhiev2007}.

Our approach is also based on derivative polynomials. This leads to a family of integer triangles satisfying a very simple recursion that allows us to calculate all derivatives of general squigonometric functions. From these, finding Taylor and MacLaurin series is only a small step.

\section{Derivative polynomials}
For the sake of generality, we take as our prototypical squigonometric function $\cq^m t \cdot \sq^n t$, which allows us to handle all squigonometric functions by appropriate choices of $n$ and $m$. For instance, the tanquent $\frac{\sq t}{\cq t}$ arises from the choice $n=1$ and $m=-1$. Taking the derivative, we have
\begin{equation}
    \frac{\diff }{\diff t}\cq^m t \cdot \sq^n t = m \cq^{m+p-1} t \cdot \sq^{n-1} t -n \cq^{m-k} t \cdot \sq^{n+p-1} t .
\end{equation}
However, we note that we can cast this in a more symmetric form
\begin{equation}
    \frac{\diff }{\diff t}\cq^m t \cdot \sq^n t = \cq^{m-1} t \cdot \sq^{n-1} t \left( n \cq^p t - m \sq^p t\right).
\end{equation}
Even though we cannot simply cancel the term inside the brackets, the Pythagorean identity does allow us to rewrite it in terms of the tanquent, so that
\begin{equation}
    \frac{\diff }{\diff t}\cq^m t \cdot \sq^n t = \cq^{m-1} t \cdot \sq^{n-1} t \cdot  \frac{n  - m \tq^p t}{1 + \tq^p t}.
\end{equation}
It turns out that this relation is generalizable to any $k$th derivative with $k\geq 0$.

\begin{theorem}
    For any $n,m \in \mathbb{Z}$, $p \geq 2$ an integer and $k\geq 0$ an integer, we have
    \begin{equation}\label{eq:symmetric_derivative_polynomial}
        \frac{\diff^k}{\diff t^k} \cq^m t \cdot \sq^n t = \cq^{m-k} t \cdot \sq^{n-k} t \frac{Q_k\big(-\tq^p t\big)}{(1+\tq^p t)^k},
    \end{equation}
    where $Q_k(u)$ is a polynomial that satisfies the recursion
    \begin{equation}\label{eq:derivative_polynomial_recursion}
        Q_{k+1}(u) = \big(n-k + (m + k(p-1))u \big) Q_k(u) + pu (1-u) Q^\prime_k(u),
    \end{equation}
    together with the initial condition $Q_0(u) = 1$.
\end{theorem}

\begin{remark}
    We define the derivative polynomials in this seemingly strange way so that all coefficients end up being nonnegative for $n,m\geq0$.
\end{remark}

\begin{proof}
    Let us for convenience write $u= -\tq^p t$, then we have
    \begin{equation*}
        \frac{\diff u}{\diff t} = -p \frac{\tq^{p-1} t }{\cq^2 t} = \frac{p u}{\cq t \cdot \sq t}.
    \end{equation*}
    Thus, differentiation of \eqref{eq:symmetric_derivative_polynomial} leads to
    \begin{align*}
        \frac{\diff^{k+1}}{\diff t^{k+1}} \cq^m t \cdot \sq^n t = &\cq^{m-k-1} t \cdot \sq^{n-k-1} t \cdot  \frac{n - k + (m-k)u}{1-u} \frac{Q_k(u)}{(1-u)^k}  \\
        &+ \cq^{m-k-1} t \cdot \sq^{n-k-1} t \cdot pu \left( \frac{Q^\prime_k(u)}{(1-u)^k} + k \frac{Q_k(u)}{(1-u)^{k+1}} \right).
    \end{align*}
    Bringing everything under one denominator, we obtain
    \begin{align*}
        \frac{\diff^{k+1}}{\diff t^{k+1}} \cq^m t \cdot \sq^n t = \frac{\cq^{m-k-1} t \cdot \sq^{n-k-1} t}{(1-u)^{k+1}} \Big( &\left(m-k + (n-k)u \right) Q_k(u) \\
        &+ pu (1-u) Q^\prime_k(u) + kpu Q_k(u) \Big).
    \end{align*}
    Grouping terms, we obtain
    \begin{align*}
        \frac{\diff^{k+1}}{\diff t^{k+1}} \cq^m t \cdot \sq^n t = \frac{\cq^{m-k-1} t \cdot \sq^{n-k-1} t}{(1+u)^{k+1}} \Big( &\left(n-k + (m + k(p-1))u \right) Q_k(u) \\
        &+ pu (1-u) Q^\prime_k(u) \Big).
    \end{align*}
    It is easy to see that the term in brackets is a polynomial, so that it must therefore be equal to $Q_{k+1}(u)$.
\end{proof}

Although the coefficients of the polynomial $Q_k(u)$ are typically not symmetric, they do exhibit a type of symmetry if we swap the highest and lowest powers. This is easily observed if we make the substitution $t^\prime = \frac{\pi_p}{2} - t$. We then obtain
\begin{equation}
\begin{aligned}
    (-1)^k\frac{\diff^k}{\diff t^k} \sq^m t \cdot \cq^n t &= \cq^{n-k}t\cdot \sq^{m-k}t \frac{Q_k(-\frac{1}{\tq^p t})}{( 1+ \frac{1}{\tq^p t} )^k} \\
    &= \cq^{n-k}t\cdot \sq^{m-k}t \frac{ \tq^{pk} t \cdot Q_k(-\frac{1}{\tq^p t})}{( 1+ \tq^p t )^k}.
    \end{aligned}
\end{equation}
From this, we may deduce that $u^k Q_k(\frac{1}{u})$ has the effect of switching $m$ and $n$, which we shall write as $m \leftrightarrow n$. Or, conversely, the derivative polynomial of $\sq^m t \cdot \cq^n t$ is given by the same derivative polynomial as $\cq^m t \cdot \sq^n t$, but with the coefficients reversed according to $j^\prime = k-j$. In particular, this implies $Q_k(u)$ is a  palindromic polynomial for $n=m$.

\subsection{A family of number triangles}

The polynomial recursion can be recast into a recursion on the coefficients. Defining the coefficients as
\begin{equation}
    Q_k(u) = \sum_{j=0}^k q^{(k)}_j u^j,
\end{equation}
it is not too hard to see that we obtain
\begin{equation}\label{eq:recursive_derivative_coefficients}
    q^{(k+1)}_j = (n-k + pj) q^{(k)}_j + \big( m + k(p-1) - p(j-1) \big) q^{(k)}_{j-1}.
\end{equation}
We should point out that the coefficients depend on $n$, $m$, $p$, $k$ and $j$. However, to avoid cumbersome notation, we only explicitly denote $k$ and $j$, as those vary over the various sums and series we will consider, while $n$, $m$ and $p$ are fixed per case. These parameters will thus be taken to be understood from the context. When we wish to make the dependence on $m$ and $n$ explicit, we will use the notation $q^{(k)}_j[m,n]$.

We note that, using the Pythagorean identity \eqref{eq:pythagorean_identity} and a bit of algebra, it is possible to show that
\begin{equation}\label{eq:derivative_polynomial_cq_sq_form}
    \frac{\diff^k}{\diff t^k} \cq^m t \cdot \sq^n t = \sum_{j=0}^k (-1)^j q^{(k)}_j \cq^{m+k(p-1)-pj} t \cdot \sq^{n-k+pj} t.
\end{equation}
We should point out that this is a homogeneous polynomial with homology $h_k = n + m + k(p-2)$. This form will also be useful later when we wish to find the MacLaurin series. It should be pointed out that the coefficient form of the recursion, together with \eqref{eq:derivative_polynomial_cq_sq_form} produces the correct derivatives for a much greater range of parameters. In fact, it holds for all $n,m,p \in \mathbb{R}$. However, we restrict ourselves here to the integer case so that the derivatives take the relatively simple form of polynomials.

We display here the table of coefficients for the 4-cosquine ($m=1$, $n=0$ and $p=4$):
\begin{center}
\begin{tabular}{cccccc}
1 &   &  &   \\
 & 1 &   &  &   \\
  & 3 &   &   \\
  & 6 & 9 &    \\
  & 6 & 81 & 18  \\
  &  & 378 & 549 & 18   \\
  &  & 1134 & 6867 & 2394   \\
\end{tabular}
\end{center}
For instance, we may obtain the $6$th derivative of $\cq_4 t$ from the table, and using \eqref{eq:derivative_polynomial_cq_sq_form} that
\begin{equation}
    \frac{\diff^6}{\diff t^6} \cq_4 t =1134 \cq_4^{11} t \sq_4^2 t - 6867 \cq_4^7t \sq_4^6 t +  2394 \cq_4^3 t \sq_4^{10} t .
\end{equation}

We will now make some small observations on a few specific coefficients in the number triangle. To ease presentation, we will employ the falling factorial, for which we will opt to use the underline notation, i.e.,
\begin{equation}
    x^{\underline{k}} = x(x-1)\cdots(x-k+1).
\end{equation}
We then have the following statements.

\begin{lemma}\label{lem:zeroth_coefficient}
    We have
    \begin{equation}
        q^{(k)}_0 = n^{\underline{k}}.
    \end{equation}
\end{lemma}

\begin{proof}
    We observe that setting $j=0$ in the recursion \eqref{eq:recursive_derivative_coefficients}, we only have the first term, such that
    \begin{equation*}
        q^{(k+1)}_0 = (n-k)q^{(k)}_0.
    \end{equation*}
    Using the fact that $q^{(k)}_0=1$, this is nothing but the recursive definition of the falling factorial.
\end{proof}

\begin{lemma}\label{lem:superdiagonal}
    We have
    \begin{equation}
        q^{(k)}_j = 0
    \end{equation}
    for all $j>k$.
\end{lemma}

\begin{proof}
    We proceed by induction. Clearly, the statement is true for $k=0$, $q^{(0)}_j=0$ for all $j>0$. Assume therefore that the statement is true for some $k$. Let $j$ be some integer such that $j > k+1$. From the recursion \eqref{eq:recursive_derivative_coefficients}, we see $q^{(k+1)}_j$ depends linearly on $q^{(k)}_j$ and $q^{(k)}_{j-1}$, but $j-1>k$, so that both are zero.
\end{proof}

\begin{lemma}\label{lem:diagonal_coefficient}
    We have
    \begin{equation}
        q^{(k)}_k = m^{\underline{k}}.
    \end{equation}
\end{lemma}

\begin{proof}
    We observe that setting $j = k+1$ in the recursion \eqref{eq:recursive_derivative_coefficients}, we only have the second term, such that
    \begin{equation*}
        q^{(k+1)}_{k+1} = (m + k(p-1) - pk)q^{(k)}_0 = (m-k)q^{(k)}_k.
    \end{equation*}
    Using the fact that $q^{(k)}_0=1$, this is nothing but the recursive definition of the falling factorial.
\end{proof}

With these Lemmas in place, we will deduce some global properties of the coefficients.

\begin{theorem}\label{thm:nonzero_coefficients}
    Let $n,m,k \geq 0$ and $p \geq 2$. Then all columns and diagonals eventually vanish. Specifically, let $j \in \mathbb{N}\cup\{0\}$ and $d \in \mathbb{Z}$, then $q^{(k)}_j = 0$ for all $k > n + pj$ and $q^{(k)}_{k-d} = 0$ for all $k > m + pd$. Moreover, all nonzero coefficients are in fact positive.
\end{theorem}

\begin{proof}
    We prove the statement that all coefficients are zero for $k > n + pj$ by induction on $j$. For the base case, we apply Lemma~\ref{lem:zeroth_coefficient} and see that for $k > n$, we have $q^{(k)}_0 = 0$. Next, fix some $j \geq 0$ and assume $q^{(k)}_j = 0$ for all $k>n+pj$. The recursion for $q^{(k)}_{j+1}$ therefore reads
    \begin{equation*}
        q^{(k+1)}_{j+1} = \big(n - k + p(j+1)  \big)q^{(k)}_{j+1}.
    \end{equation*}
    We then set $k = n + (j+1)p$, which is strictly larger than $n+pj$, since $p \geq 2$, so that the recursion yields $q^{(k+1)}_{j+1} = 0$. Hence, plugging this back into the recursion, we find $q^{(k+1)}_{j} = 0$ for all $k>n+pj$.

    Next, we prove that diagonals eventually vanish. We do this by investigating $q^{(k)}_{k-d}$ by induction on $d$. Lemma~\ref{lem:superdiagonal} tells us that $q^{(k)}_{k-d} = 0$ for all $d<0$. As our base case $d=0$, we use Lemma~\ref{lem:diagonal_coefficient} and see that $q^{(k)}_k = 0$ for all $k>m$. Next, we assume that $q^{(k)}_{k-d} = 0$ for all $k>m+pd$. We use the recursion to find coefficient $q^{(k+1)}_{k+1-(d+1)} = q^{(k+1)}_{k-d}$, which reads
    \begin{equation*}
        q^{(k+1)}_{k-d} = \big( m + k(p-1) - p(k-d-1) \big) q^{(k)}_{k-d-1},
    \end{equation*}
    We write $q^{(k)}_{k-d-1} = q^{(k)}_{k-(d+1)}$ and simplify to find
    \begin{equation*}
        q^{(k+1)}_{k+1-(d+1)} = \big( m - k + p(d+1) \big) q^{(k)}_{k-(d+1)}.
    \end{equation*}
    Thus, setting $k = m+p(d+1) > m+pd$, we see the coefficient vanishes. Moreover, by the recursion we see that for all $k>m+p(d+1)$, we have $q^{(k+1)}_{k-d}=0$. To finish this part of the proof, we note that $j = k-d$, meaning $d= k-j$. The condition therefore reads $k>m+p(k-j)$, or $pj > m + k(p-1)$.

    Lastly, we again use induction to show that all nonzero coefficients are positive. Therefore, let $k$ and $j$ be such that $pj \leq m + k(p-1)$ and $k \leq n+pj$. It is not hard to see that this is a nonempty set, for instance by isolating $k$ as $ \frac{pj - m}{p-1} \leq k \leq n + pj$. The upper bound grows faster than the lower bound with increasing $k$, since $p \geq 2$. Next, we use induction over $k$. Our base case is $q^{0)}_0 = 1$, which is positive. The column condition tells us that $ n - k + pj \geq 0$, while the diagonal condition tells us $m+k(p-1)-p(j-1) \geq p >0$. Consequently, by the recursion \eqref{eq:recursive_derivative_coefficients}, we see that $q^{(k+1)}_j>0$.
\end{proof}

It should be pointed out that the above Theorem has a special implication for $p=2$, in that there is only one single nonzero coefficient. Of course, we already know this from basic calculus. Theorem~\ref{thm:nonzero_coefficients} can be used to show this in a very roundabout way. Indeed, setting $p=2$, the Theorem tells us that $\frac{k-n}{2} \leq j \leq \frac{k+m}{2}$. The fact that $j$ must be integer yields $\lceil\frac{k-n}{2}\rceil \leq j \leq \lfloor \frac{k+m}{2} \rfloor$. Now, setting $m=0$ and $n=1$, both bounds become the same and we obtain $j = \lfloor \frac{k}{2} \rfloor$. Similarly, if $m=1$ and $n=0$, we obtain $j = \lceil \frac{k}{2} \rceil$.

\subsection{Algebraic values of squigonometric functions}
The fact that all derivatives of squigonometric functions expressed in terms of polynomials has some interesting consequences for the values they take.

\begin{lemma}\label{lem:number_of_zero_roots}
    Let $m,n,k\geq0$ be integers. Then, the polynomial $Q_k(u)$ has degree  $k - \lceil \frac{k-m}{p} \rceil$. Furthermore, it has a root of multiplicity $\lceil \frac{k-n}{p} \rceil$ at $u=0$.
\end{lemma}

\begin{proof}
    We apply Theorem~\ref{thm:nonzero_coefficients} to see that the lowest power of $u$ occurring in $Q_k(u)$ is the minimal value of $j$ such that $j_\mathrm{min} \geq \frac{k-n}{p}$, so that $j_\mathrm{min} = \lceil \frac{k-n}{p} \rceil$. The highest power of $u$ occurs for $j=k-d$, where $d_\mathrm{min} \geq \frac{k-m}{p}$, which is $d_\mathrm{min} = \lceil \frac{k-m}{p} \rceil$, so that $j_\mathrm{max}= k-d_\mathrm{min} = k-\lceil \frac{k-m}{p}\rceil$.

    Using this argument, we can see that we can factor out a power of $u^{j_\mathrm{min}}$ from $Q_k(u)$, which tells us $u=0$ is a root of multiplicity $\lceil \frac{k-n}{p} \rceil$. When we factor out this root, we end up with a polynomial of degree $k - \lceil \frac{k-n}{p} \rceil - \lceil \frac{k-m}{p} \rceil$ with a nonzero constant coefficient. Thus, all roots of the remaining polynomial are nonzero.
\end{proof}

We present now the following Theorem which establishes real-rootedness of the derivative polynomials. The proof is rather long and technical, but simple in concept: we divide out all the roots of $Q_{k+1}(u)$ at $u=0$ and evaluate at the roots $Q_k(u)$. This produces a sequence of points where $Q_{k+1}(u)$ has an alternating sign. The subtleties arise due to the multitude of possible cases implied by the previous Lemma. As such, we urge the reader to skip the proof on a first reading.

\begin{theorem}\label{thm:real_rootedness}
    Let $m,n,k\geq0$ be integers. Then, the sequence of polynomials $Q_k(u)$ has only real roots. Moreover, all nonzero roots are strictly negative and simple. Finally, $Q_{k+1}(u)$ interlaces $Q_k(u)$, i.e., let $u_0 < u_1 < \ldots,u_{R_{k+1}}$ be the nonzero roots of $Q_{k+1}(u)$ and $v_0 < v_1 < \ldots,v_{R_k}$ be the nonzero roots of $Q_k(u)$, then
    \begin{equation}
        u_0 < v_0 < u_1 < v_1 \ldots < u_{R_{k+1}-1} < v_{R_{k}-1} < 0.
    \end{equation}
\end{theorem}

\begin{proof}
    First, we observe that by Theorem~\ref{thm:nonzero_coefficients}, all coefficients of $Q_k(u)$ are positive. As such, there are no positive real roots. Define $j_k = \lceil \frac{k-n}{p} \rceil$ and $J_k = k-\lceil \frac{k-m}{p} \rceil$. From the previous lemma, we know that $Q_k(u)$ has degree $J_k$ and $R_k = J_k - j_k$ nonzero roots. Moreover, by Descartes' rule of signs, $Q_k(u)$ has $J_k - j_k - 2D$ real roots for some integer $D\geq0$. We wish to show that $D=0$.

    Next, we proceed by induction on $k$. As the base case, we can use any $Q_k(u)$ that has two nonzero coefficients. From Theorem~\ref{thm:nonzero_coefficients}, these must be consecutive, thus $Q_k(u) = a x^{j_k} + b x^{j_k+1}$ for some $a,b>0$. Hence, there is a single nonzero root and it is negative and simple. We therefore assume that $Q_k(u)$ has $J_k-j_k$ negative simple roots.

    Let us define $\tilde{Q}_k(u) = u^{-j_k}Q_k(u)$, so that $\tilde{Q}_k(u)$ only has negative simple roots that are equal to the negative simple roots of $Q_k(u)$. We have
    \begin{equation*}
        Q^\prime_k(u) = u^{j_k}\tilde{Q}^\prime_k(u) + j_k u^{j_k-1}\tilde{Q}_k(u).
    \end{equation*}
    Thus, we see that $\tilde{Q}_{k+1}(u)$ therefore satisfies
    \begin{align*}
        u^{j_{k+1}} \tilde{Q}_{k+1}(u) = \big(m-k + (n+k(p-1))u\big)u^{j_k}\tilde{Q}_k(u) + p u^{j_k} (1-u)\tilde{Q}^\prime_k(u)\\
        + j_k u^{j_k-1}(1-u)\tilde{Q}_k(u).
    \end{align*}

    We now evaluate $\tilde{Q}_{k+1}(u)$ at the negative simple roots of $Q_k(u)$. As such, we can also divide by the $u^{j_{k+1}}$ factor on the left-hand side. Let $0 \leq r < R_k$, so that $v_r<0$ and it follows that
    \begin{align*}
        \sgn \tilde{Q}_{k+1}(v_r) &= \sgn p v_r^{j_k - j_{k+1}} (1-v_r)\tilde{Q}^\prime_k(v_r)\\
        &= (-1)^{j_k - j_{k+1}} \sgn \tilde{Q}^\prime_k(v_r).
    \end{align*}
    Now, by the induction hypothesis, all the roots of $\tilde{Q}_k(u)$ are simple, which means that $\tilde{Q}^\prime_k(v_r)$ has an alternating sign as a function of $r$. Moreover, $\tilde{Q}_k(u)$ has only positive coefficients, so that we know $\tilde{Q}^\prime_k(v_{R_k-1})>0$. Furthermore, $\tilde{Q}_{k+1}(u)$ also has only positive coefficients, so that $\tilde{Q}_{k+1}(0)>0$.
    
    We should distinguish two cases: if $j_{k+1} = j_k$, we have $J_k - j_k+1$ sign changes in $\tilde{Q}_{k+1}(u)$; while if $j_{k+1}=j_k$, we observe $J_k - j_k$ sign changes. By continuity, in between each sign change $\tilde{Q}_{k+1}(u)$ has a real root, so that the interlacing property is established. We summarize
    \begin{equation*}
        R_{k+1} \geq \begin{cases}
            R_k & \text{ if } j_{k+1} = j_k,\\
            R_k-1 & \text{ if } j_{k+1} = j_k+1.
        \end{cases}
    \end{equation*}

    Now, since for $\tilde{Q}_k(u)$, we have $R_k = J_k - j_k$ by the induction hypothesis, we find that
    \begin{equation*}
        J_{k+1} - j_{k+1} = \begin{cases}
            R_k + 1 & \text{ if } J_{k+1} = J_k+1 \text{ and } j_{k+1} = j_k,\\
            R_k & \text{ if } J_{k+1} = J_k \text{ and } j_{k+1} = j_k,\\
            R_k & \text{ if } J_{k+1} = J_k+1 \text{ and } j_{k+1} = j_k+1,\\
            R_k-1 & \text{ if } J_{k+1} = J_k \text{ and } j_{k+1} = j_k+1.
        \end{cases}
    \end{equation*}
    Next, applying Descartes' rule of sign changes to $\tilde{Q}_{k+1}(u)$, we see that $R_{k+1} = J_{k+1}-j_{k+1} - 2D$, where $D \geq 0$ is some integer. We now inspect case-by-case what the bounds imply about $D$.

    First, we treat $J_{k+1}-j_{k+1} = R_k$ and note that we can bound both cases where this occurs as $R_{k+1} \geq R_k -1$. But Descartes says $R_{k+1} = R_k - 2D$, so that $D=0$.

    Next, we have $J_{k+1} - j_{k+1} = R_k+1$, where we have the bound $R_{k+1} \geq R_k$. Descartes says $R_{k+1} = R_k +1 - 2D$, so again we find $D=0$.

    Finally, we have $J_{k+1} - j_{k+1} = R_k - 1$, for which we know $R_{k+1} \geq R_{k-1} - 1$. Descartes says $R_{k+1} = R_k - 1 - 2D$, so that for the last time, we have $D=0$.
\end{proof}

\begin{prop}\label{prop:algebraic_values}
    Let $k \geq 0$, then squigonometric functions take on algebraic values at points where the $k$th derivative vanishes. In particular, if $u$ is a root $Q_k(u)$, then
    \begin{subequations}\label{eq:cq_and_sq_algebraic}
    \begin{align}
        \cq t &= \pm\left(1-u \right)^{-\frac{1}{p}},\\
        \sq t &= \pm\left( \frac{u}{u-1} \right)^\frac{1}{p}.
    \end{align}
    \end{subequations}
\end{prop}

\begin{proof}
    Observe that if the $k$th derivative vanishes, then $u=-\tq^p t$ is a root of $Q_k(u)$. Of course, $Q_k(u)$ is a polynomial with integer coefficients, so that $u$ is algebraic. Consequently, any squigonometric function can be expressed as an algebraic function of $u$.

    Let us assume without loss of generality that we are working in the first quadrant. We can express the cosquine as $\cq t = (1+\tq^p)^{-\frac{1}{p}}$, and the squine as $\sq t = \left( \frac{\tq^p t}{1 + \tq^p t}\right)^\frac{1}{p}$. Using $u=-\tq^p t$, we find \eqref{eq:cq_and_sq_algebraic}.
\end{proof}

We should point out here that, even though $Q_k(u)$ pretty much always has a root at $u=0$, this does not imply that any derivative vanishes at $k=0$. This has to do with the prefactor $\cq^{m-k} t \sq^{n-k}t$. In particular, the sine carries a factor $ u^\frac{n-k}{p}$, so that if $\frac{n-k}{p}$ is an integer, the roots at $u=0$ are exactly cancelled.

As an example of Proposition~\ref{prop:algebraic_values}, consider the third derivative of the 4-cosquine. The derivative polynomial is $9 u^2 + 6 u$, which has roots at $u = 0$ and $ u = -\frac{2}{3}$. Consequently, the 4-cosquine takes the value of $(\frac{3}{5})^\frac{1}{4}$ at the nontrivial point where $ \frac{\diff^3 }{\diff t^3}\cq_4 t = 0$.

To give a slightly more complicated example, we can consider the fourth derivative of the 6-squine, which has as its derivative polynomial $60u^3 + 425u^2 + 100u$. This polynomial has roots at $0$ and $- \frac{85\pm \sqrt{6265}}{24} $. As such, the squine takes on the values $(\frac{125 \pm \sqrt{6265}}{234})^{\frac{1}{6}}$ at the nontrivial points where $ \frac{\diff^4 }{\diff t^4}\sq_6 t = 0$.

Finally, we consider the general case for some $p\geq 2$ that $n,m> 0$ and consider the first derivative of $f(t) = \cq^m t \sq^n t$. The derivative polynomial is $n+mu$, which has $u=-\frac{n}{m}$ as a root. We therefore find that $\cq t = (\frac{m}{n+m})^\frac{1}{p}$ and $\sq t = (\frac{n}{n+m})^\frac{1}{p}$ where $f^\prime(t) = 0$. This leads to the rather amusing expression that
\begin{equation}
    \cq^m t \sq^n t = \left(\frac{m^m \cdot n^n}{(n+m)^{(n+m)}}\right)^\frac{1}{p},
\end{equation}
where the derivative vanishes.

The combination of the results from this section leads to another curious fact about squigonometric functions: if $p>2$, this number of points where the $k$th derivative vanishes in $(0,\frac{\pi_p}{2})$ tends to infinity as $k$ approaches infinity. This is a very strange property of squigonometric functions that is not shared by trigonometric functions. Even if we consider the cosquine for any $p>2$, then there are a large number of points where the $k$th derivative vanishes if $k$ is large. For the cosine, on the other hand, there are in fact at most two points in $[0,\frac{\pi}{2}]$ where the $k$th derivative vanishes.

\section{Power series and continued fractions}
We now turn to the question of power series and other infinite expressions. Most useful will be the MacLaurin series, the power series around $t=0$. We also consider Taylor series around other points and continued fractions.

\subsection{MacLaurin series}
The expression in terms of squines and cosquines, \eqref{eq:derivative_polynomial_cq_sq_form} will be most useful here. The only thing we need to do is realize that, as $t\to 0$, we have $\cq t \to 1$ and $\sq t \to 0$, so that the derivative of our general function $\cq^n t \cdot \sq^m t$ at $t=0$ can be expressed as $f_k(1,0)$. This limit always exists for $k\geq 0$ if $n,m \geq 0$, i.e., negative powers cannot be created by differentiation. Thus,
\begin{equation}\label{eq:maclaurin_series_coefficients}
    \left. \frac{\diff^k}{\diff t^k} \cq^m t \cdot \sq^n t \right|_{t \to 0} = \sum_{j=0}^k q^{(k)}_j \delta_{0,n-k+pj},
\end{equation}
where we have used the Kronecker delta. The MacLaurin series coefficient of order $k$ is nonzero only if there is an integer $j$ such that $n + pj = k$.

In general, we therefore find
\begin{equation}\label{eq:squine_maclaurin}
    \sq t = \sum_{j=0}^\infty (-1)^j S_j \frac{t^{pj+1}}{(pj+1)!},
\end{equation}
The coefficients $S_j$ are very easily found from the number triangles introduced above and as such are in fact integers. For instance, the first few coefficients for the 4-squine are given by $1, 18, 14364, 70203672,\ldots$.

Note that in a practical setting, we could compute the MacLaurin series coefficients for squine and cosquine simultaneously, as they can both be taken from the same number triangle. This follows from the earlier observation that switching $m \leftrightarrow n$ is equivalent to switching $j$ to $k-j$. For clarity, we include $m$ and $n$ into the notation for the coefficients, we find
\begin{equation}
    S_j =  q^{(pj+1)}_{j}[0,1] = q^{(pj+1)}_{1+(p-1)j}[1,0] .
\end{equation}
Applying the same argument to the cosquine ($m=1$, $n=0$), there is only a nonzero coefficient if there is an integer $k$ such that $k=jp$, i.e.,
\begin{equation}\label{eq:cosquine_maclaurin}
    \cq t =  \sum_{j=0}^\infty (-1)^j C_j \frac{t^{pj}}{(pj)!}.
\end{equation}
Again due to the symmetry of the number triangle, we have
\begin{equation}
    C_j =  q^{(pj)}_{j}[1,0] = q^{(pj)}_{(p-1)j}[0,1].
\end{equation}
The first few coefficients for the 4-cosquine are given by $1, 6, 2268, 7434504,\ldots$

In order to compute the MacLaurin series coefficients in a manner that is more numerically stable, we introduce the scaled coefficients $\alpha^{(k)}_j = \frac{q^{(k)}_j}{k!}$, where we incorporate the factorial into the recursion. It is not too hard to see that these satisfy
\begin{equation}\label{eq:scaled_recursive_coefficients}
    \alpha^{(k+1)}_j = \frac{n-k+pj}{k+1} \alpha^{(k)}_j + \frac{m+k(p-1)-p(j-1)}{k+1} \alpha^{(k)}_{j-1}.
\end{equation}
Since the $q^{(k)}_j$ are all integers, $\alpha^{(k)}_j$ are all rational numbers.

\subsection{Taylor series around other points}

Having access to all derivatives, and not just the MacLaurin series coefficients, we can set up a Taylor series around any point of our choosing. Perhaps the most interesting point to do this for is $t=\frac{\pi_p}{4}$, where $\cq \frac{\pi_p}{4} = \sq \frac{\pi_p}{4} = \frac{1}{\sqrt[p]{2}}$. The coefficients for the Taylor series are then easily determined as
\begin{equation}
    f_k := \frac{1}{k!}\left. \frac{\diff^k}{\diff t^k} \cq^n t \sq^m t \right|_{t = \frac{\pi_p}{4}} = 2^{-\frac{h_k}{p}} \sum_{j=0}^k (-1)^j \alpha^{(k)}_j,
\end{equation}
where $h_k = n+m+k(p-2)$. To reiterate, the coefficients $\alpha^{(k)}_j$ depend on the fixed parameters $n$, $m$, and $p$, asides from $k$ and $j$; we omit them for brevity. The Taylor series is then, of course, given by
\begin{equation}
    \cq^n t \sq^m t = \sum_{k=0}^\infty f_k (t-\tfrac{\pi_p}{4})^k.
\end{equation}

\subsection{Continued fractions and numerical stability}
Recall that the MacLaurin series of squigonometric functions typically have a finite radius of convergence. We have seen that the MacLaurin coefficient of power $t^{n+pj}$ is given by $\frac{F_j}{(n+pj)!}$, where $F_j$ is an integer. These two facts combined show us that the sequence of integers $F_j$ grows roughly as fast as the factorial. We wish to practically compute possibly hundreds of MacLaurin series coefficients, hence, in order to prevent numerical issues, we need to represent the information in $F_j$ differently.

Note:
\begin{equation}
    F_j = p! q_j^{(n+p(j-1))} + \big( m + n(p-1) + p(p-2)(j-1) \big) (p-1)! F_{j-1}.
\end{equation}

One way was already presented in the form of \eqref{eq:scaled_recursive_coefficients}. We now present a second way that removes the growth behaviour entirely. To do so, we first consider writing $\frac{F_j}{(n+pj)!}$ as a product of $j+1$ factors, i.e.,
\begin{equation}
    \frac{F_j}{(n+pj)!} = a_0 \cdot a_1 \cdots a_j,
\end{equation}
where $a_j$ are a sequence of rational numbers. Thus, we have $a_0 = \frac{F_0}{n!}=1$, and in general for $j \geq 1$, we have
\begin{equation}\label{eq:factor_definition}
    a_j = \frac{F_j}{F_{j-1}(n+pj)^{\underline{p}}}.
\end{equation}
In order to find a recursion for the factor coefficients $a_j$, we define each nonzero $q^{(k)}_j$ as the product of factors $\beta_0^{(k)}n!$, $\beta_1^{(k)}(n+p)^{\underline{p}}$, up to $\beta_j^{(k)} (n+pj)^{\underline{p}}$. For $k>n+pj$, we simply fix $\beta^{(k)}_j = \beta^{(k-1)}_j = a_j$. For $k+1\leq n+pj$, we apply the recursion \eqref{eq:recursive_derivative_coefficients} to find
\begin{equation}
    \beta^{(k+1)}_j = \frac{\beta_0^{(k)}\beta_1^{(k)}\cdots\beta_{j-1}^{(k)}}{\beta_0^{(k+1)}\beta_1^{(k+1)}\cdots\beta_{j-1}^{(k+1)}}\left( (n-k+pj)\beta^{(k)}_j + \frac{m  + k(p-1) - p(j-1)}{(n+pj)^{\underline{p}}} \right).
\end{equation}

Practically, these numbers can be computed quite efficiently. The first few factors for the 4-cosquine are given by $1, \frac{1}{4}, \frac{9}{40}, \frac{149}{540}, \ldots$

From the finite radius of convergence, we have
\begin{equation}
    R_p^p = \lim_{j\to\infty} \frac{\left(n+ jp \right)!}{\left(n+(j-1)p\right)!} \frac{F_{j-1}}{F_j} = \lim_{j\to\infty}  \frac{1}{a_j} .
\end{equation}
This allows us to estimate that
\begin{equation}
    a_j \approx \frac{1}{R_p^p}.
\end{equation}
Hence, in the limit, this sequence converges to a constant that is close to unity. In this way, we are not very likely to run into numerical issues, as all our constants are of the same order of magnitude and for large $j$ will be more or less constant. Note that we can also see that
\begin{equation}
    \pi_p = \lim_{j\to\infty} \frac{4\cos( \tfrac{\pi}{p} )}{ a_j^{\frac{1}{p}}}.
\end{equation}
We will actually use this identity as an initial guess when we compute $\pi_p$.

After obtaining the factors $a_j$, we have
\begin{equation}\label{eq:factorial_type_expansion}
    f(t) = t^n \sum_{k=0}^\infty \prod_{j=1}^k \left( - a_j t^p \right).
\end{equation}
We should note that, if we replace $a_j$ by its asymptotic estimate in \eqref{eq:factorial_type_expansion}, we obtain what is essentially a geometric series.

In this particular form, we can also obtain a continued fraction for $f(t)$ very easily by means of Euler's continued fraction identity. We find
\begin{equation}
    f(t) = \cfrac{t^n}{1+ \cfrac{a_1 t^p}{1-a_1 t^p + \cfrac{a_2 t^p}{1 - \dots}}}.
\end{equation}
We can then use \eqref{eq:factor_definition} and apply an equivalence transformation to find the continued fraction in terms of the integer coefficients as
\begin{equation}
    \cq^m t\cdot \sq^n t = \cfrac{F_0 t^n}{n! + \cfrac{ F_1 \,n!\,  t^p}{F_0(n+p)^{\underline{p}} - F_1 t^p + \cfrac{ F_0 F_2 (n+p)^{\underline{p}}   t^p}{F_1(n+2p)^{\underline{p}} - \dots}}}.
\end{equation}
It should be noted that in the case of $p=2$ and $m=0$, $n=1$ or $m=1$ and $n=0$, all integer coefficients are identically 1, leading to the familiar continued fractions for the sine and cosine.

\section{Explicit formulae for MacLaurin coefficients}

The polynomial recursion \eqref{eq:derivative_polynomial_recursion}, or in its coefficient form \eqref{eq:recursive_derivative_coefficients}, is probably the most efficient way to practically compute the derivative polynomials. However, it is possible to provide explicit expressions for all derivative polynomial coefficients, which we shall demonstrate here.

In order to do so, we shall first introduce a matrix formulation of the recursion. Fix some $K \in \mathbb{N}$, which is the highest derivative we consider for the moment. Listing the coefficients $q^{(k)}_j$ in a vector $\mathbf{q}^{(k)}$, the recursion can clearly be cast as $\mathbf{q}^{(k)} = T_k \mathbf{q}^{(k)}$. We define the matrices
\begin{equation}
    S = \begin{pmatrix}
        0 & &\\
        1 & 0 & \\
         & 1 & \ddots
    \end{pmatrix}
\end{equation}
and
\begin{equation}
    D = \begin{pmatrix}
        0 & &  \\
         & 1 & \\
        &  & 2  
    \end{pmatrix}.
\end{equation}
We can see that $S$ is a shift matrix and $D$ is a matrix that is related to differentiation, in particular, $DS$ implements differentiation on the vector space of polynomial coefficients. Note also that the commutator is given by $DS-SD = [D,S] = S$. It can then be checked that $T_k$ is given by
\begin{equation}
    T_k = (n-k)I + \big(m+p+k(p-1))S + pD(I-S).
\end{equation}

Perhaps the most obvious way to split up this matrix is into a constant and a part that multiplies with $k$, i.e., let 
\begin{subequations}
    \begin{align}
        A &= I - (p-1)S,\\
        B &= nI + pD(I-S) + (m+p)S.
    \end{align}
\end{subequations}
We then have
\begin{equation}
    T_k = B-kA.
\end{equation}
The recursion can now be easily expressed as
\begin{equation}\label{eq:matrix_factorials}
\begin{aligned}
    \mathbf{q}^{(k+1)} &= \big(B - kA \big) \big(B - (k-1)A \big) \cdots \big(B - A \big) B\mathbf{q}^{(0)}.
\end{aligned}
\end{equation}
We can express this more succinctly by using left-multiplying product notation, i.e., for $k\geq 0$, we use
\begin{equation}
    \prod_{l=0}^k X_l = \begin{cases}
        I, & \text{ if } l=0,\\
        X_k X_{k-1} X_{k-2}\ldots X_1 X_0, & \text{ otherwise}.
    \end{cases}
\end{equation}
Thus,
\begin{equation}\label{eq:matrix_factorial_beta}
    \mathbf{q}^{(k+1)} = \left(\prod_{l=0}^k T_l \right) \mathbf{q}^{(0)} = \left(\prod_{l=0}^k (B-lA) \right) \mathbf{q}^{(0)}.
\end{equation}
Hence, the derivative polynomial coefficients are related to matrix-valued falling factorials. Such objects have been studied previously, for instance by Tirao \cite{tirao_hypergeometric_matrix} and Schlosser \cite{Schlosser_2018}. Their research has focussed around hypergeometric series. To give a simple example, we can express $(1+z)^B$ analogously to binomial series using matrix factorials. The novelty here seems to be that matrix $A$ is not the identity matrix. Regardless, it shows that there is a connection in general between the derivative polynomials of squigonometric functions and the Stirling numbers.

In order to obtain an explicit expression for any component of $\mathbf{q}^{(k)}$, though, we need to split $T_k$ in another way: by diagonal and subdiagonal. Doing so allows us to prove the following result.

\begin{theorem}\label{thm:explicit_coefficients_binary}
Let $p,k,n,m \in \mathbb{Z}$ with $k \geq 0 $ and $p \geq 2$, then
    \begin{equation}\label{eq:beta_explicit}
        q^{(k)}_j = \sum_{\substack{\mathbf{i} \in \{0,1\}^k \\ |\mathbf{i}|=j }} \prod_{l=0}^{k-1} \left( n-l + i_l\big(m-n+pl\big) + p(1-2i_l)   \sum_{r=0}^{l-1} i_r \right),
    \end{equation}
    where $\mathbf{i} = (i_0,i_1,\ldots,i_k) \in \{0,1\}^k$ and $|\mathbf{i}|=\sum_{r=0}^{k-1} i_r$.
\end{theorem}

\begin{proof}
    Fix some $0 \leq l \leq k-1$ and consider the $l$th factor of the matrix product \eqref{eq:matrix_factorial_beta}, $T_l$. Observe that we have two key facts: first, we always have $\mathbf{q}^{(0)} = \mathbf{e}_0$; second, we note that $T_l$ has only two nonzero diagonals. In particular, we can write $T_l = D_l + S_l$, where $D_l$ is a diagonal matrix and $S_l$ only has nonzero values on the first subdiagonal. We can see that
    \begin{align*}
        D_l &= (n-l)I + pD,\\
        S_l &= \left(\big( m + p + l(p-1) )\big) I - p D \right) S.
    \end{align*}
    We then have $\mathbf{q}^{(k)} = \prod_{l=0}^{k-1} (D_l+S_l) \mathbf{e}_0$.
    
    Observe next, that in expanding the product of matrices, each term consists of products of diagonal matrices and shift matrices. Multiplying from the left with $\mathbf{e}_0$, each term in the expansion therefore produces a vector that has only a single nonzero component.

    Using the binary vector $\mathbf{i}$, we can express the expansion as
    \begin{equation*}
        \prod_{l=0}^{k-1} (D_l+S_l) = \sum_{\mathbf{i}\in \{0,1\}^k} \prod_{l=0}^{k-1} \big( (1-i_l) D_l + i_l S_l  \big).
    \end{equation*}
    Clearly, each term in the sum produces a vector that has a nonzero component only at index $\sum_{l=0}^{k-1} i_l$. Conversely, given some $j$, we can find $q^{(k)}_j$ as the sum over all terms labelled with $|\mathbf{i}|=i_0+i_1+\ldots+i_{k-1} = j$. Thus,
    \begin{equation*}
        q^{(k)}_j = \mathbf{e}_j^T \sum_{\substack{\mathbf{i} \in \{0,1\}^k \\ |\mathbf{i}|=j }} \left(\prod_{l=0}^{k-1} \big( (1-i_l) D_l + i_l S_l  \big) \right) \mathbf{e}_0.
    \end{equation*}
    
    Next, by the associativity of matrix multiplication, each term in the sum can be calculated by consecutive matrix-vector products. Let us call the intermediate result $\mathbf{v}_l$ after $l$ factors have been multiplied, thus $\mathbf{v}_l = \prod_{r=0}^l \big( (1-i_r) D_r + i_r S_r  \big) \mathbf{e}_0$. Due to our splitting, we note that each $\mathbf{v}_l$ has only a single nonzero component. In this way, the whole product for a single fixed $\mathbf{i}$ reduces to a multiplication of scalars only. We can keep track of the index of the nonzero component in the following way. Consider $\mathbf{v}_{l-1}$, then if $i_l = 1$, the index is increased by 1, while if $i_l = 0$ it stays the same. Therefore, we find that after $l$ multiplications, $\mathbf{v}_l$ only has a nonzero component at index $\sum_{r=0}^{l} i_r$.

    Next, we see that there are two types of diagonal multiplications, the identity and $D$. If we multiply by the identity, the result is independent of the index of the nonzero component. If we multiply by $D$, the result is multiplied by the index of the nonzero component. Using these rules, we find \eqref{eq:beta_explicit} after some basic manipulation.
\end{proof}

\subsection{Factor sequences}

Although we have stated the explicit formula in terms of binary sequences, there is in fact a very large number of binary sequences that lead to a zero term in the sum. Indeed, if any factor in a particular product vanishes, the entire term vanishes. To illustrate the point, it is possible to show that there is only one single binary sequence that leads to a nonzero factor sequence for the sine or cosine. We shall give a complete characterisation of binary sequences that lead only to nonzero terms.

First, we introduce the factor sequence $f_l[\mathbf{i}]$, as
\begin{equation}
    f_l[\mathbf{i}] = n-l + i_l (m-n + pl) +p(1-2i_l)\sum_{r=0}^{l-l} i_r.
\end{equation}
Next, we switch to slightly different parametrization of the binary sequence, where we record only the positions of the ones. Suppose there are $E$ ones in a binary sequence of length $k$. We then say that these ones occur at $0 \leq l_0 < l_1 < \ldots < l_{E-1} \leq k-1$. Carefully going through the cases, we can see that
\begin{equation}
    \sum_{r=0}^{l-1} i_r = \begin{cases}
        E , &\text{ if } l > l_{E-1},\\
        \vdots &  \\
        r , &\text{ if } l_{ r - 1 } <  l \leq l_{r},\\
        \vdots & \\
        0 , & \text{ if } l \leq  l_0.
        \end{cases}
\end{equation}
These cases also carry over to $f_l[\mathbf{i}]$, which are further compounded by another $E$ cases for when $l=l_r$. We find
\begin{equation}\label{eq:factor_split_cases}
    f_l[\mathbf{i}] = \begin{cases}
        n + p E - l, &\text{ if } l > l_{E-1},\\
        \vdots &  \\
        m + l(p-1) - pr , & \text{ if } l =  l_r,\\
        n + pr -l   , &\text{ if } l_{ r-1 } <  l < l_r,\\
        \vdots & \\
        n-l , & \text{ if } l <  l_0.
        \end{cases}
\end{equation}

This splitting of cases based on the location of the ones allows us to provide a complete characterisation of the binary sequences that lead to nonzero factors. In fact, it is simply a matter of isolating $l$ and negating the conditions.

\begin{prop}\label{thm:nonzero_factor_characterisation}
    Let $p$ and $Z$ be integers with $p\geq2$ and $E\geq 0$. Let $n$ and $m$ be integers with at least one of them nonzero. Let $k = n + pE$ and $j = E$.

    A binary sequence produces nonzero factors if and only if the following two conditions are met. For all $r=0,1,\ldots,Z-1$:
    \begin{subequations}
    \begin{equation}\label{eq:nonzero_equality_condition}
        m + l_r(p-1) \neq pr.
    \end{equation}
    For all $r=1,1,\ldots,Z-1$:
    \begin{equation}\label{eq:nonzero_inequality_condition}
        l_r \leq n + pr \quad \text{ or } \quad l_{r-1} \geq n + pr.
    \end{equation}
    If $m \geq 0$, we also require
    \begin{equation}\label{eq:nonzero_inequality_zero}
        l_0 \leq n.
    \end{equation}
    \end{subequations}
\end{prop}

\subsection{Rethinking the factorial}
To summarize the above discussion, we present here the following alternative to Theorem~\ref{thm:explicit_coefficients_binary}.

\begin{cor}\label{cor:explicit_factorial_form}
    Let $n,m \in \mathbb{Z}$ and consider the function $f(t) = \cq^n t \sq^m t$. Then, $f(t)$ has a MacLaurin series given by $f(t) = f_n t^n + f_{n+p} t^{n+p} + f_{n+2p} t^{n+2p}+\ldots$. Let $k = n + pj$, the coefficients are given by
    \begin{equation}\label{eq:explicit_factorial_form}
    f_k = (-1)^j\sum_{\mathcal{L}} \frac{ m^{\underline{l_0}}   }{ k^{\underline{l_{j-1}+1}}  } \prod_{r=0}^{j-1}\big( m + l_r(p-1) - pr \big) \prod_{l=l_{r-1}+1}^{l_r - 1} (n + pr -l),
\end{equation}
where $\mathcal{L} = \{0 \leq l_0 < \ldots < l_{j-1} \leq k-1\}$ and we define $l_{-1} = l_0 - 2$.
\end{cor}

We stress here that $n$ and $m$ are allowed to be any integers, including negative numbers. As such, Corollary~\ref{cor:explicit_factorial_form} provides MacLaurin coefficients for any squigonometric (or trigonometric) function. For completeness, we set $m=1$ and $n =0$ to obtain the series coefficients for the cosquine as
\begin{equation}
    c_k = (-1)^j\sum_{0 < l_1 < \ldots < l_{j-1} \leq k-1} \frac{1 }{k^{\underline{l_{j-1}+1}}} \prod_{r=1}^{j-1}\big( 1+ l_j(p-1) - pr  \big) \prod_{l=l_{r-1}+1}^{l_r - 1} (pr -l),
\end{equation}
where $k=jp$. Note that we also included the fact that $l_0 = 0$. Setting $n=1$ and $m=0$, we obtain the series coefficients for the squine,
\begin{equation}
    s_k = (-1)^j \sum_{1 < l_1 < \ldots < l_{j-1} \leq k-1}  \frac{p-1}{k^{\underline{l_{j-1}+1}}} \prod_{r=1}^{j-1}\big( l_r(p-1) - pr \big) \prod_{l=l_{r-1}+1}^{l_r - 1} (1 + pr -l),
\end{equation}
where $k=1+pj$. Here we included the fact that $l_0=1$.

It is interesting to note that \eqref{eq:explicit_factorial_form} contains objects that seem to generalize the factorial beyond just rising and falling. Indeed, the two products both contain elements that are both rising and falling simultaneously, as $l_r$ is a strictly increasing sequence of integers.

\subsection{Bounds on the number of nonzero factor sequences}
Note that the conditions on $k$ and $j$ from Theorem~\ref{thm:nonzero_factor_characterisation} correspond exactly to the MacLaurin series coefficients. Thus, we are able to completely characterise the binary sequences that produce nonzero factors for the MacLaurin series. Next, we ask how many of them there are, at least roughly. This seems to be a rather difficult counting problem, so we will not be solving it here. For instance, for the 4-cosquine ($n=0$, $m=1$, $p=4$), we have for $j = 1,2,3,4,5,6$ that the number of sequences is $1,3,15,91,611,4373$. This seems to be growing exponentially at least. 

To get a handle on this number, we will construct some relatively simple bounds. Rather trivially, it being a subset of all binary sequences of length $k$, we may say the number is bounded from above by $2^k = 2^{n+pj}$, which at least shows an exponential bound. Of course, a slightly more sophisticated bound is $\binom{n+pj}{j}$, which yields $1,5,36,286,2380,20349,\ldots$, using again the 4-cosquine as an example. We are mostly interested in a lower bound, which is constructed in the following lemma.

\begin{lemma}
    Let $n$ and $m$ be nonzero integers such that $n+m>0$. The number of binary sequences from Theorem~\ref{thm:nonzero_factor_characterisation} that produce nonzero factors is bounded from below by $(n+1)(p-1)^{j-1}$.
\end{lemma}

\begin{proof}
    Let $1 \leq r \leq j-1$, and consider the one labelled $r$. We observe $l_{r-1}+1 \leq l_r \leq n+pr$. Hence, the range the one can occupy is certainly larger than $n+(r-1)p + 1 \leq l_r \leq n + pr$, which are $p$ positions. Furthermore, we can naively imagine the equality constraint simply takes away one possibility. Therefore, if we estimate that we have $p-1$ choices for $l_r$, we are certainly underestimating the actual number of possibilities.

    Next, we consider $l_0$ and distinguish several cases. If $n>0$, then we have $n+1$ possibilities. If, on the other hand, $n=0$, then $l_0 = 0$ is the only possibility. In both cases, we thus have $n+1$ possibilities.

    Multiplying everything together, we find there are $(n+1)(p-1)^{j-1}$ total possibilities.
\end{proof}

Applying the lower bound again to the 4-cosquine, we obtain $1,3,9,27,81,234$. Comparing this, we see that the number of nonzero factor binary sequences definitely grows much faster. However, what is important to note here is that the lower bound increases exponentially. As such, if we want to practically compute the MacLaurin series coefficients, the recursion \eqref{eq:recursive_derivative_coefficients} is the way to go.

\section{Computing squigonometric functions}
We will now propose a method to compute the values of squigonometric functions over an interval up to a given tolerance, say $\varepsilon>0$. The only method available up until now, as far as we know, was to compute the value of the inverse functions and invert those. This usually involves numerical quadrature or the evaluation of a hypergeometric series and root finding.

However, with the above theory, it becomes possible to evaluate squine and cosquine with MacLaurin and Taylor series, which is potentially much less computationally intensive, or at the very least conceptually easier. Ultimately, after some preparatory work has been done, the squine and cosquine can be evaluated as a (piecewise) high-degree polynomial.

\begin{figure}
    \centering
    \includegraphics[width=0.5\linewidth]{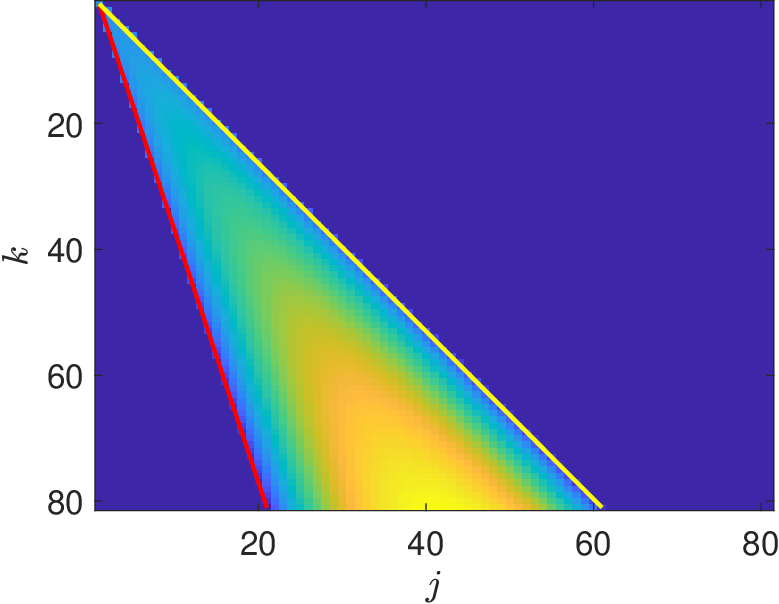}
    \caption{Plot of $\ln|q^{(k)}_j|$ for $n=0$, $m=1$ and $p=4$. The red line indicates the location of the MacLaurin coefficients of the cosquine, and the yellow line indicates the location of the squine MacLaurin coefficients.}
    \label{fig:log_abs_beta}
\end{figure}

\subsection{Choice of interval}
We should note that the interval $[0,1]$ is sufficient to evaluate the squigonometric functions over $\mathbb{R}$. Indeed, by symmetry, we only need the interval $[0,\frac{\pi_p}{4}]$, while we know $\pi_p<4$. As such, if we know the squine and cosquine to sufficient accuracy on $[0,1]$, symmetry and periodicity allows us to extend this to all of $\mathbb{R}$.

\subsection{Single polynomial or piecewise?}
One approach is to simply determine a MacLaurin series that produces a sufficiently accurate approximation across the entire interval $[0,1]$. Once the series coefficients have been computed, these can be stored for later use, requiring only a polynomial evaluation to compute the squine or cosquine. Considering that our prototypical squigonometric function $\cq^m t \cdot \sq^n t$ has only nonzero coefficients for $k = n + pj$, where $j = 0,1,\ldots$, we can store only the nonzero coefficients and use a specialized version of Horner's method for the evaluation.

An alternative to this is to split up $[0,1]$ into subintervals and use a piecewise approximation of lower order. We can daisy-chain these together, starting from $t=0$. We can compute the squigonometric functions to sufficient accuracy to find the Taylor coefficients of the next interval. However, doing so introduces a lot of nonzero Taylor coefficients, while it also necessitates computing all $q^{(k)}_j$. As we will explain, this is not needed for the MacLaurin series. As such, it is probably more efficient to use MacLaurin series.

\subsection{Algorithmic optimization}

If we take a close look at the recursion \eqref{eq:recursive_derivative_coefficients}, we should notice that, in order to compute a particular coefficient $q^{(k)}_j$, we need only values $q^{(k-1)}_j$ and $q^{(k-1)}_{j-1}$. However, we know that MacLaurin coefficients correspond to $k = n + pj$ for $j=0,1,\ldots$ If we know in advance the highest-order coefficient we want to know, say $K = n + pJ$, then we only need a single array to work with.

Indeed, all coefficients with $j > n + pJ$ can be ignored. The evaluation in going from $k$ to $k+1$ can be done from right to left, so that we can update in place. This also means we do not even need a separate array to store the MacLaurin coefficients.

The number of nonzero coefficients we need can be estimated by the finite radius of converge. From the ratio test, we get the estimate $f_k \approx \frac{f_0}{R_p^k} = \frac{1}{R_p^k}$. Working over the interval $[0,1]$, we see that for some given tolerance $\varepsilon$, we want $f_K \approx \varepsilon$. Thus, we can estimate $J \approx - \frac{\ln \varepsilon}{p\ln R_p} - \frac{n}{p} < - \frac{\ln \varepsilon}{p\ln R_p}$, so that we can use
\begin{equation}\label{eq:max_terms_accuracy_estimate}
    J  = \left\lceil - \frac{\ln \varepsilon}{p\ln R_p} \right\rceil.
\end{equation}
For $p=4$ and double precision, $\varepsilon = 2^{-53}$, this yields approximately 32 terms.

We present here a short overview of the resulting algorithm, see Algorithm~\ref{alg:inplace_maclaurin}. In order to compute $\cq^m t \cdot \sq^n t $ on the entire real line, we simply run the algorithm twice, the second time with $m$ and $n$ switched.

\begin{algorithm}[H]
    \SetAlgoLined
    \KwData{$m,n\in\mathbb{Z}$, $J\in\mathbb{N}$}
    \KwResult{$\mathbf{f}$: MacLaurin coefficients of $\cq^m t \cdot \sq^n t$}
    Set $\mathbf{f}$ to be a vector of zeros with length $J+1$\;
    $f_0 \gets 1$\;
    $K \gets n + pJ$\;
    \For{$k=0$ : $(K-1)$}{
        $j_{k+1} \gets \max\{ \lceil \frac{k+1-n}{p} \rceil$\,0 \} \;
        $J_{k+1} \gets \min\{k+1 - \lceil \frac{k+1-m}{p} \rceil$\,J\} \;
        \For{$j=J$ : $-1$ : $(j_{k+1}+1)$}{
            $f_j \gets \frac{n-k+pj}{k+1}f_j + \frac{m+k(p-1)-p(j-1)}{k+1}f_{j-1}$\;
        }
        \eIf{ $\mathrm{mod}(k-n,p) > 0$ \KwSty{or} $j_k=0$  }{
            $f_{j_{k+1}} \gets \frac{n-k+pj_{k+1}}{k+1}f_{j_{k+1}}$\;
            }{
            $f_{j_{k+1}} \gets \frac{n-k+pj_{k+1}}{k+1}f_{j_{k+1}} + \frac{m+k(p-1)-p(j_{k+1}-1)}{k+1}f_{j_{k+1}-1}$\;
        }
    }
\caption{In-place computation of MacLaurin coefficients.}
\label{alg:inplace_maclaurin}
\end{algorithm}

\subsection{Specialized Horner's method}
After obtaining the nonzero MacLaurin coefficients, we can evaluate the approximating polynomial using what is essentially Horner's method. After running Algorithm~\ref{alg:inplace_maclaurin}, we obtain a vector $\mathbf{f}$ with components $(f_0,f_1,\ldots,f_J)$. The function $f(t) = \cq^m t \cdot \sq^n t$ is then approximated by
\begin{equation}
    f(t) \approx t^n \sum_{j=0}^J (-1)^j f_j t^{pj}.
\end{equation}
Following Horner's method, we define $b_J(t) = f_J$, so that we can recursively compute
\begin{equation}
    b_{j-1}(t) = f_{j-1} - b_j(t) t^p.
\end{equation}
The polynomial is then given by $t^n b_0(t)$.

\subsection{Factorial-type coefficients}
Another option is to work with the factorial-type coefficients defined in \eqref{eq:factor_definition}. As explained in the relevant section, these coefficients converge to the constant $\frac{1}{R_p^p}$. Hence, when the number of terms in the power series is very large, it can be advantageous to work with these for numerical stability.

We note that \eqref{eq:factorial_type_expansion} can be evaluated very efficiently. We define $b_0(t) = t^n$ and then recursively compute
\begin{equation}
    b_j(t) = -a_j \, t^p\, b_{j-1}(t),
\end{equation}
so that $f(t) = \sum_{j=0}^\infty b_j(t)$. This last sum can, of course, be computed in the same loop as the recursion for the $b_j(t)$'s. This method is preferable over alternatives, since the loop can be stopped when sufficient accuracy has been obtained. We simply monitor the size of the update $|b_j(t)|$ and see if it falls below the tolerance.

\section{Applications}

To start with, here we present in Table~\ref{tab:cqsq_maclaurin_coeffs} the 34 nonzero MacLaurin series coefficients that yield double precision approximations of $\cq t$ and $\sq t$ for $p=4$ on $[0,1]$. As pointed out earlier, this is sufficient to determine them everywhere on $\mathbb{R}$ due to symmetry and periodicity. We are unsure whether these coefficients have been presented elsewhere before.

\begin{table}[]
    \centering
    \begin{tabular}{c|c||c|c}
$c_{0}$  & $1$ & $s_{1}$  & $1$ \\
$c_{4}$  & $-0.25$ & $s_{5}$  & $-0.15$ \\
$c_{8}$  & $0.05625$ & $s_{9}$  & $3.958333333333333\cdot 10^{-2}$ \\
$c_{12}$  & $-1.552083333333333\cdot 10^{-2}$ & $s_{13}$  & $-1.127403846153846\cdot 10^{-2}$ \\
$c_{16}$  & $4.551382211538462\cdot 10^{-3}$ & $s_{17}$  & $3.351438772624435\cdot 10^{-3}$ \\
$c_{20}$  & $-1.378667338329563\cdot 10^{-3}$ & $s_{21}$  & $-1.023074271776018\cdot 10^{-3}$ \\
$c_{24}$  & $4.262157319358032\cdot 10^{-4}$ & $s_{25}$  & $3.178544597827111\cdot 10^{-4}$ \\
$c_{28}$  & $-1.336246164457652\cdot 10^{-4}$ & $s_{29}$  & $-9.999541389962689\cdot 10^{-5}$ \\
$c_{32}$  & $4.232450123540253\cdot 10^{-5}$ & $s_{33}$  & $3.175297169293426\cdot 10^{-5}$ \\
$c_{36}$  & $-1.35111622811414\cdot 10^{-5}$ & $s_{37}$  & $-1.015610312254756\cdot 10^{-5}$ \\
$c_{40}$  & $4.339818399729758\cdot 10^{-6}$ & $s_{41}$  & $3.267168144275697\cdot 10^{-6}$ \\
$c_{44}$  & $-1.400932280792665\cdot 10^{-6}$ & $s_{45}$  & $-1.055981761709645\cdot 10^{-6}$ \\
$c_{48}$  & $4.541006938457347\cdot 10^{-7}$ & $s_{49}$  & $3.426394557128631\cdot 10^{-7}$ \\
$c_{52}$  & $-1.477032838858058\cdot 10^{-7}$ & $s_{53}$  & $-1.115450251730353\cdot 10^{-7}$ \\
$c_{56}$  & $4.818461802551568\cdot 10^{-8}$ & $s_{57}$  & $3.641564287838735\cdot 10^{-8}$ \\
$c_{60}$  & $-1.575906447463121\cdot 10^{-8}$ & $s_{61}$  & $-1.191751121195521\cdot 10^{-8}$ \\
$c_{64}$  & $5.165517597599798\cdot 10^{-9}$ & $s_{65}$  & $3.908490898639448\cdot 10^{-9}$ \\
$c_{68}$  & $-1.696454484883442\cdot 10^{-9}$ & $s_{69}$  & $-1.284247026995244\cdot 10^{-9}$ \\
$c_{72}$  & $5.581090331746424\cdot 10^{-10}$ & $s_{73}$  & $4.226807242319111\cdot 10^{-10}$ \\
$c_{76}$  & $-1.838922930079956\cdot 10^{-10}$ & $s_{77}$  & $-1.393233097731723\cdot 10^{-10}$ \\
$c_{80}$  & $6.067464431015257\cdot 10^{-11}$ & $s_{81}$  & $4.59851156387092\cdot 10^{-11}$ \\
$c_{84}$  & $-2.004434127494118\cdot 10^{-11}$ & $s_{85}$  & $-1.519627134628728\cdot 10^{-11}$ \\
$c_{88}$  & $6.62928857700319\cdot 10^{-12}$ & $s_{89}$  & $5.027300346286176\cdot 10^{-12}$ \\
$c_{92}$  & $-2.194770387531149\cdot 10^{-12}$ & $s_{93}$  & $-1.664825711881474\cdot 10^{-12}$ \\
$c_{96}$  & $7.273112274138255\cdot 10^{-13}$ & $s_{97}$  & $5.518261217435007\cdot 10^{-13}$ \\
$c_{100}$  & $-2.412276992591459\cdot 10^{-13}$ & $s_{101}$  & $-1.83064057439664\cdot 10^{-13}$ \\
$c_{104}$  & $8.007193281522936\cdot 10^{-14}$ & $s_{105}$  & $6.07774971493359\cdot 10^{-14}$ \\
$c_{108}$  & $-2.659828303685806\cdot 10^{-14}$ & $s_{109}$  & $-2.019278258811832\cdot 10^{-14}$ \\
$c_{112}$  & $8.841456068111032\cdot 10^{-15}$ & $s_{113}$  & $6.71337044606192\cdot 10^{-15}$ \\
$c_{116}$  & $-2.940831036310465\cdot 10^{-15}$ & $s_{117}$  & $-2.233345646081413\cdot 10^{-15}$ \\
$c_{120}$  & $9.787544448612465\cdot 10^{-16}$ & $s_{121}$  & $7.434024574502131\cdot 10^{-16}$ \\
$c_{124}$  & $-3.259252131579342\cdot 10^{-16}$ & $s_{125}$  & $-2.475873000498906\cdot 10^{-16}$ \\
$c_{128}$  & $1.085894046080356\cdot 10^{-16}$ & $s_{129}$  & $8.250004847723769\cdot 10^{-17}$ \\
    \end{tabular}
    \caption{First 32 coefficients of the MacLaurin series for $\cq t$ ($c_k$) and $\sq t$ ($s_k$). These coefficients are available in a Matlab file upon request.}
    \label{tab:cqsq_maclaurin_coeffs}
\end{table}

\subsection{Computing $\pi_p$}
Being able to compute the squigonometric functions up to a given accuracy on all of $\mathbb{R}$ means we can very easily determine $\pi_p$. Indeed, we can simply apply any root-finder to, e.g., $\cq t - 2^{-\frac{1}{p}}$. This function has a root at $\frac{\pi_p}{4}$. After determining $J$ via \eqref{eq:max_terms_accuracy_estimate}, we employ Newton's method using the initial guess
\begin{equation}
    \frac{\pi_p}{4} \approx \cos\left( \tfrac{\pi}{p} \right) a_J^{-\frac{1}{p}}.
\end{equation}
This estimate becomes more accurate as $p$ becomes larger, since $J$ grows with increasing $p$. All values of $\pi_p$ are found within 4 iterations for $\varepsilon = 2^{-53}$.

\begin{table}[h]
    \centering
    \begin{tabular}{c|c|c|c}
          & value  & \# nnz MacLaurin \\ \hline 
          $\pi_3$ & $3.533277500570900$ & 22 \\
          $\pi_4$ & $3.708149354602744$  & 34 \\
          $\pi_5$ & $3.800600555953747$ & 46 \\
          $\pi_6$ & $3.855242593319996$ & 58 \\
          $\pi_7$ & $3.890174737625689$ & 69 \\
          $\pi_8$ & $3.913843287813181$ & 81 \\
          $\pi_9$ & $3.930614378886605$ & 92 \\
          $\pi_{10}$ & $3.942927897810032$ & 103
    \end{tabular}
    \caption{Several values of $\pi_p$ and the number of nonzero MacLaurin coefficients needed to reach machine precision.}
    \label{tab:pip_values}
\end{table}

As $p$ increases, the radius of convergence shrinks to $1$ from above, and $\frac{\pi_p}{4}$ approaches $1$ from below. Thus, we are operating more closely to the boundary of the convergence region. As such, we need increasingly many MacLaurin coefficients to compute the approximations accurately. To illustrate this, we have also listed the number of nonzero coefficients needed in Table~\ref{tab:pip_values}.

It is hard to quantify the relative performance of computing $\pi_p$ via the MacLaurin series and a root-finder, as compared to other methods. For instance, computing $\pi_p$ via a hypergeometric series converges linearly, i.e., a number of terms that grows linearly with the number of digits of accuracy. However, then we would only have the number $\pi_p$. We also need a linear number of MacLaurin coefficients, but then we have a squigonometric function up to a given accuracy on the entire real line. Computing $\pi_p$ is a happy by-product, as we could apply any root-finder and thereby obtain any convergence rate we desire. The limiting factor is the accuracy as determined by the number of MacLaurin coefficients.

\subsection{Computing the Beta function at rational values}
Poodiack and Wood state the following identity involving squigonometric functions and the Beta function,
\begin{equation}
    \int\limits_0^\frac{\pi_p}{2} \cq^m t \cdot \sq^n t \diff t = \tfrac{1}{p} \Beta\left( \tfrac{m+1}{p},\tfrac{n+1}{p}\right).
\end{equation}
This is a generalization of Euler's result for $p=2$. Using the symmetry of the squine and cosquine around $\frac{\pi_p}{4}$, we can also write this as
\begin{equation}
    \int\limits_0^\frac{\pi_p}{4} \cq^m t \cdot \sq^n t + \cq^n t \cdot \sq^m t \diff t = \tfrac{1}{p} \Beta\left( \tfrac{m+1}{p},\tfrac{n+1}{p}\right).
\end{equation}
This last form allows us to exploit the fact that the MacLaurin series are convergent in this interval.

Let us call $f(t) = \cq^m t \cdot \sq^n t$ and $g(t) = \cq^n t \cdot \sq^n t$, with MacLaurin coefficients $f_k$ and $g_k$ respectively. These can be computed up to a prescribed accuracy using Algorithm~\ref{alg:inplace_maclaurin}. Let us say we have $f_{n+pj}$ and $g_{m+pj}$ for $j = 1,2,\ldots,J$. Moreover, if the constant $\pi_p$ is not stored, it can also be extracted from the table of coefficients by realizing that $f(t) - 2^{-\frac{n+m}{p}}$ or $g(t) - 2^{-\frac{n+m}{p}}$ both have a root at $\frac{\pi_p}{4}$. After that, we simply compute
\begin{equation}
    \Beta\left( \tfrac{m+1}{p},\tfrac{n+1}{p}\right) \approx p \sum_{j=0}^J \left( \frac{f_{n+pj}}{n+pj+1} \left(\frac{\pi_p}{4}\right)^{n + pj +1} + \frac{g_{m+pj}}{m+pj+1} \left(\frac{\pi_p}{4}\right)^{m + pj +1} \right),
\end{equation}
where the error is on the order of the tolerance $\varepsilon$.

\section{Conclusion}
We have presented here a new approach for the calculation of derivatives of squigonometric functions based on derivative polynomials. This led, in turn, to the introduction of a new family of integer triangles. Eventually, we were able to compute MacLaurin and Taylor series for all squigonometric functions.

Our investigations led us to uncover some very interesting properties of squigonometric functions. We showed the derivative polynomials are real-rooted. As such, squigonometric functions take on algebraic values where the derivatives vanish. Furthermore, we were able to determine the exact number of roots of the derivative polynomials, specifying the multiplicity of the root at zero and showing that all other roots are negative. The roots of consecutive polynomials are interlacing. The number of points where the $k$th derivative of a squigonometric function with $p>2$ vanishes, grows linearly with $k$.

In terms of practical results, we have presented a recursion that allows the computation of MacLaurin series coefficients for the squine and cosquine for general $p$ up to arbitrary order. Algorithm~\ref{alg:inplace_maclaurin} represents a very efficient algorithm that computes a fixed number of nonzero MacLaurin series coefficients. Together with an error estimate, this allowed us to compute polynomial approximations that are accurate up to any given precision of $\mathbb{R}$. For completeness, we have listed a number of MacLaurin series coefficients for the 4-squine and 4-cosquine.

As an application, we computed several values of $\pi_p$, which are very easy to compute once the polynomial approximations are found. All cases needed only 4 iterations of Newton's method, starting at $t_0=1$, to compute $\pi_p$. We have also demonstrated a method to compute values of the Beta function at rational values up to any given accuracy.

\bibliographystyle{amsplain}

\end{document}